\documentclass[12pt]{article}
\usepackage{amsmath,amssymb,amsthm,lineno,physics}
\usepackage[utf8]{inputenc}
\usepackage{tikz, fancyvrb,hyperref}
\usetikzlibrary{calc}

\newcommand{\Reals}{{\mathbb R}}
\newcommand{\y}{\text{\bf y}}
\newcommand{\f}{\text{\bf f}}
\newcommand{\FF}{\text{\bf F}}
\newcommand{\E}{\text{\bf E}}
\renewcommand{\tr}{^{\text{\tiny\sf T}}}
\newcommand{\m}{\phantom{-}}

\def\squash{1}

\input trees.tex

\def\T#1{\scalebox{\squash}{\enspace\raisebox{-1mm}{\csname Tree#1\endcsname}\enspace}}
\def\S#1#2{\scalebox{\squash}{\enspace\raisebox{-1mm}{\csname Stump#1#2\endcsname}\enspace}}

\newtheorem{theorem}{Theorem}
\theoremstyle{remark}
\newtheorem{remark}{Remark}

\begin{document}

\title{Isomeric trees and the order of Runge--Kutta methods}
\author{John C. Butcher and Helmut Podhaisky}

\date{}
\maketitle

\begin{abstract}
The conditions for a Runge--Kutta method to be of order $p$ with $p\ge 5$ for
a scalar non-autonomous problem are a proper subset of the order conditions for a vector problem. 
Nevertheless, Runge--Kutta methods that were derived historically only for scalar problems happened to be of the same order 
for vector problems. 
We relate the order conditions for scalar problems to factorisations of the Runge--Kutta trees into
``atomic stumps'' and enumerate those conditions up to $p=20$.
Using  a special search procedure
over unsatisfied order conditions, new Runge--Kutta methods of ``ambiguous orders'' five and six are constructed.
These are used to verify the validity of the results.
\end{abstract}


\section{Introduction}\label{sect:int}
The pioneers in the theory of Runge--Kutta methods, Runge \cite{Runge}, Heun
\cite{Heun}, Kutta \cite{Kutta} and Nystr\"om \cite{Nystrom} studied accuracy
and order questions using a generic initial value problem
\begin{equation}\label{eq:scalar}
	y'(x) = f(y(x),x), \quad y(x_0)=y_0, \qquad y(x)\in \Reals.
\end{equation}
However, 
in recent times, a more general test problem has been used:
\begin{equation}\label{eq:vector}
	\y'(x) = \f(\y(x)), \quad \y(x_0)=\y_0, \qquad \y(x)\in \Reals^d,\quad d\ge 1.
\end{equation}
It is known that the order conditions for (\ref{eq:scalar}) are a proper subset
of those for (\ref{eq:vector}) although, up to $p=4$, the conditions for order
$p$ are identical in the two cases.  For $p=5$, there are 17 conditions for the
vector case.
For a tableau, for an $s$ stage explicit method,
\begin{equation*}
	\begin{array}{c|c}
		c & A    \\
		\hline
		  & b\tr
	\end{array}=
			\begin{array}{c|ccccccccc}
				0       & 0                                                              \\
				c_2     & a_{21}                                                         \\
				c_3     & a_{31}    & a_{32}                                             \\
				\vdots  & \vdots    & \vdots    & \ddots                                 \\
				c_{s-1} & a_{s-1,1} & a_{s-1,2} & \dots  & a_{s-1,s-2}                   \\
				c_{s}   & a_{s,1}   & a_{s,2}   & \dots  & a_{s,s-2}   & a_{s,s-1}       \\
				\hline
				        & b_1       & b_2       & \dots  & b_{s-2}     & b_{s-1}   & b_s
			\end{array} 
            ,
\end{equation*}
two of these conditions are
\begin{subequations}
\begin{align}
	\sum b_i c_i a_{ij} a_{jk}c_k  & =  \tfrac1{30}, \label{XXX}\\
	\sum b_i a_{ij} c_j a_{jk} c_k & =  \tfrac1{40}, \label{YYY}
\end{align}
\end{subequations}
but, in the scalar case, these are replaced by the single condition
\begin{equation}
	\sum b_i (c_i+c_j) a_{ij} a_{jk}c_k =  \tfrac7{120}. \label{ZZZ}
\end{equation}
In discussing the relationship between the vector and the scalar cases, it will be convenient to use a generic problem which is both high-dimensional and non-autonomous:
\begin{equation}\label{eq:non_auto}
	\y'(x) = \f(\y(x),x), \quad \y(x_0)=\y_0, \qquad \y(x)\in \Reals^d.
\end{equation}
A preliminary announcement of the results of this paper was given in \cite{2018-butcher-axioms}.
Here, we will evaluate the order conditions up to $p=20$.
Known methods of ambiguous order five were found using traditional algebraic techniques.  In this paper we use a systematic search routine over as-yet unresolved order conditions. 
This makes it possible to derive new methods of ambiguous orders five and six respectively.
The paper is organised as follows. 
In Section \ref{sect:trees}, we will
review the basics of B-series and trees.  This includes a consideration of
the order conditions for a vector based initial-value problem. 
This is followed by Section
\ref{sect:stumps} in which trees are related to atomic stumps. 
The formal Taylor series expansion for scalar problems is 
derived in Section~\ref{sect:order}, for comparison with the known vector problem solution. 
In Section ~\ref{sect:o56}, a search procedure is outlined for systematically  deriving methods  satisfying subsets of the order conditions.  This technique was used to find two new methods of ``ambiguous order'' five and six respectively.  Tests reported in Section~\ref{sect:exp} give numerically-computed orders close to their expected values.

\section{Trees and B-series}\label{sect:trees}

The formal Taylor series for the solution to (\ref{eq:vector}) can be written 
in the ``B-series" form
\begin{equation}\label{eq:DE}
	\y(x_0+h) = \y_0+ \sum_{t\in T} \frac{h^{|t|}}{\sigma(t)t!} \FF(t)(\y_0),
\end{equation}
where $t\in T$ is a generic (rooted) tree, $|t|$ is the order, $\sigma(t)$ is the symmetry, $t!$ is the factorial of $t$ and $\FF(t)(\y_0)$ is the elementary differential of $\f$ evaluated at $\y_0$.

Similarly, the Taylor series for the numerical solution using a Runge--Kutta method
is given by
\begin{equation}\label{eq:RK}
\y_1 = \y_0+ \sum_{t\in T} \frac{h^{|t|}\Phi(t)}{\sigma(t)} \FF(t)(\y_0),
\end{equation}
where $\Phi(t)$ denotes the ``elementary weight".  By comparing (\ref{eq:RK}) with
(\ref{eq:DE}) we deduce the  conditions for order $p$
\begin{equation}
    \Phi(t) = \frac1{t!}, \qquad |t| \le p.
\end{equation}
The various tree-dependent quantities up to order 4 are given as
\begin{equation*}
\begin{array}{c|cccccccc}
t & \T{1} & \T{2} & \T{3} & \T{4} & \T{5} & \T{6} & \T{7} & \T{8}\\
\hline
|t| & 1&2&3&3&4&4&4&4\\
\sigma(t )&1&1&2&1&6&1&2&1\\
t! &1&2&3&6&4&8&12&24\\
\FF(t) & \f &\f'\f &\f''\f\f & \f'\f'\f &\f^{(3)}\f\f\f 
& \f''\f\f'\f & \f'\f''\f\f & \f'\f'\f'\f\\
\Phi(t) & b\tr 1 & b\tr c & b\tr c^2 & b\tr Ac &
b\tr c^3& b\tr cAc & b\tr A^2c & b\tr Ac^2
\end{array} 
\end{equation*}
Further details are given in \cite{bseriesbook}.
B-series are also often used in the analysis of numerical methods when \eqref{eq:vector} assumes a special structure.
For example, the order conditions for exponential integrator can be expressed using bi-coloured trees, cf.\ \cite{MR2182810}.

Trees can be written recursively in a number of different ways, including a  system which matches the structure of the corresponding elementary differentials. 
The single tree with order $1$ is denoted by $\tau$ and $\tau_m$ will denote a prefix operator on a sequence of $m$ trees.
Thus $t=\tau_m t_1 t_2 \cdots t_m$ will represent the tree
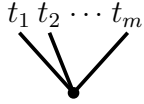
\[
\begin{tikzpicture}[black,x=8mm,y=8mm,radius=1.5pt,line width=1.4pt]
\draw[fill=black] circle;
\draw(0,0)to(-0.9,1);\node () at (-0.9,1.3) {$t_1$};
\draw(0,0)to(-0.4,1);\node () at (-0.4,1.3) {$t_2$};
\draw(0,0)to(0.9,1);\node () at (0.9,1.3) {$t_m$};
\node () at (0.25,1.3) {\small$\dots$};
\end{tikzpicture}
\]
for which the corresponding elementary weight is 
\[
	\FF(t)=\f^{(m)} \FF(t_1)  \FF(t_2) \cdots  \FF(t_m) .
\]

\section{Stumps, atomic stumps and isomeric trees}\label{sect:stumps}
A stump is formed from a tree by replacing one or more leaves by un-filled valencies.
In particular, an atomic stump, is a structure of the form
\begin{equation}
s_{mn} = \tau_{m+n}\tau^m =
\raisebox{-5mm}{\begin{tikzpicture}[black,x=9mm,y=9mm,radius=1.5pt]
\draw[fill=black](0,0) circle;
\draw[line width=1.4pt](0,0)to(-0.95,1); \draw[line width=1pt,fill=black](-0.95,1) circle;
\draw[line width=1.4pt](0,0)to(-0.15,1); \draw[line width=1pt,fill=black](-0.15,1) circle;
\draw[line width=1.4pt](0,0)to(0.15,1); \draw[line width=1pt,fill=white](0.15,1) circle;
\draw[line width=1.4pt](0,0)to(0.95,1); \draw[line width=1pt,fill=white](0.95,1) circle;
\node () at (-0.55,1) {\scriptsize$\dots$};
\node () at (0.55,1) {\scriptsize$\dots$};
\draw[<->,thick](-0.95,1.2)to (-0.15,1.2);
\draw[<->,thick](0.95,1.2)to (0.15,1.2);
\node () at (-0.55,1.4) {\footnotesize$m$};
\node () at (0.55,1.4) {\footnotesize$n$};
\end{tikzpicture}},
\end{equation}
where \tikz{\draw[fill=white,radius=2pt,very thick](0,0) circle;} denotes an valency. 
Examples of products of atomic stumps, resulting in trees, are
\begin{alignat*}{4}
	s_{21} s_{20}               & =\S41   \S30          &  & =  \T{20},                    \\
	s_{02} s_{10} s_{01} s_{20} & = \S32 \S20 \S21 \S30 &  & =  \S32 \S20 \T{7} = \T{61}.
\end{alignat*}
 All trees can be written as products of atomic stumps. For example,
\begin{align*}
	\T{4} & = \S21 \S20,    \\
	\T{6} & = \S31 \S20,    \\
	\T{7} & = \S21 \S30,     \\
	\T{8} & = \S21 \S21 \S20.
\end{align*}
 Isomeric trees are trees which have the same atomic factors, in a different order. If $t_1$ and $t_2$ are such trees, then we write $t_1\sim  t_2$.
 
Isomers do not appear until  order $5$, where the two trees
\begin{align*}
	t_{12} = \T{12} & =  \S31 \S21  \S20, \\
	t_{15} = \T{15} & =  \S21 \S31 \S20,
\end{align*}
occur.  
Hence, $ t_{12} \sim  t_{15}$. 
The numbering of  $t_{12}$  $t_{15}$, and other trees throughout the paper, are the same as the numbering adopted in \cite{bseriesbook}.

\def\squash{0.8}
For orders $5$, $6$, $7$, all isomeric classes are shown in
Table \ref{tab:iso567}.
\begin{table}[htp]
\def\arraystretch{1.3}
\caption{\label{tab:iso567}Isomers of order 5, 6 and 7}
\begin{center}
\begin{tabular}{@{}ccc}
atomic stumps & isomers & isomeric classes  \\
{\bf Order 5} \hfill \hfill\\[-2mm]
\S20 \S21 \S31 & 
\T{12} $\sim$ \T{15}
& $\{t_{12},t_{15}\}$ \\
{\bf Order 6} \hfill\hfill\\[-2mm]
	\S20 \S21 \S41      & \T{21}$\sim$\T{30}             & $\{t_{21}, t_{30}\}$         \\
	\S20 \S20 \S21 \S32 & \T{28}$\sim$\T{33}             & $\{t_{28}, t_{33}\}$         \\
	\S20 \S21 \S21 \S31 & \T{26}$\sim$\T{32}$\sim$\T{35} & $\{t_{26}, t_{32}, t_{35}\}$ \\
	\S21 \S30 \S31      & \T{25}$\sim$\T{31}             & $\{t_{25}, t_{31}\}$\\
{\bf Order 7}\hfill\hfill\\[-2mm]
    \S21 \S31 \S40           & \T{54}$\sim$\T{71}                         & $\{t_{54}, t_{71}\}$                 \\
    \S20 \S21 \S21 \S41      & \T{46}$\sim$\T{69}$\sim$\T{78}             & $\{t_{46}, t_{69}, t_{78}\}$         \\
    \S20 \S21 \S51           & \T{41}$\sim$\T{67}                         & $\{t_{41}, t_{67}\}$                 \\
    \S20 \S21 \S31 \S31      & \T{52}$\sim$\T{55}$\sim$\T{72}             & $\{t_{52}, t_{55}, t_{72}\}$         \\
    \S20 \S21 \S30 \S32      & \T{61}$\sim$\T{64}$\sim$\T{75}             & $\{t_{61}, t_{64}, t_{75}\}$         \\
    \S21 \S30 \S41           & \T{45}$\sim$\T{68}                         & $\{t_{45}, t_{68}\}$                 \\
    \S20 \S31 \S41           & \T{44}$\sim$\T{50}                         & $\{t_{44}, t_{50}\}$                 \\
    \S20 \S20 \S21 \S42      & \T{48}$\sim$\T{70}                         & $\{t_{48}, t_{70}\}$                 \\
    \S21 \S21 \S30 \S31      & \T{56}$\sim$\T{73}$\sim$\T{79}             & $\{t_{56}, t_{73}, t_{79}\}$         \\
    \S20 \S21 \S21 \S21 \S31 & \T{57}$\sim$\T{74}$\sim$\T{80}$\sim$\T{83} & $\{t_{57}, t_{74}, t_{80}, t_{83}\}$ \\
    \S20 \S20 \S21 \S21 \S32 & \T{62}$\sim$\T{65}$\sim$\T{76}$\sim$\T{81} & $\{t_{62}, t_{65}, t_{76}, t_{81}\}$ \\
    \S20 \S20 \S31 \S32      & \T{53}$\sim$\T{60}                         & $\{t_{53}, t_{60}\}$
\end{tabular}
\end{center}
\end{table}
Let $T_p$ denote the set of trees with order $p$ and $U_p$ the set of 
isomeric classes with order $p$, so that $\bigcup  U_p = T_p$.  From Table
\ref{tab:iso567}, we can write
\begin{align*}
U_5 &=\big\{ \{t_{12},t_{15}\}\big\}\\
U_6 &= \big\{\{t_{21}, t_{30}\}, \{t_{28}, t_{33}\}, \{t_{26}, t_{32}, t_{35}\},\{t_{25}, t_{31}\} \big\}\\
U_7 &=\big\{\{t_{54},t_{71}\},\{t_{46},t_{69},t_{78}\},\{t_{41},t_{67}\},\{t_{52},t_{55},t_{72}\},\{t_{61},t_{64},t_{75}\},\\
&\qquad \{t_{45},t_{68}\},\{t_{44},t_{50}\},\{t_{48},t_{70}\},\{t_{56},t_{73},t_{79}\},\{t_{57},t_{74},t_{80},t_{83}\},
\{t_{53},t_{60}\}\big\}
\end{align*}

Computationally, the isomeric classes $U_p$ can be found by (i) generating all trees to order $p$ using the algorithm 3 from \cite{bseriesbook} and (ii) grouping them according to their products of stumps.

\section{Order conditions for scalar and vector problems}\label{sect:order}
We will consider the elementary differential corresponding to the stump
$s_{mn}$.  For (\ref{eq:vector}), $\FF(s_{mn})$ is the $n$-linear operator
$\f^{(m+n)} \f^m$ and for (\ref{eq:non_auto}) this becomes
\begin{equation*}
\begin{bmatrix}
\displaystyle\sum_{i=0}^m {m\choose i} (\partial_x^{m-i} \f^{\,(i+n)})\f^{\,i}\\
0
\end{bmatrix}:=
\begin{bmatrix}
\widetilde\FF(s_{mn})\\
0
\end{bmatrix}.
\end{equation*}
If $t=s_{m_1n_1} s_{m_2n_2} \cdots s_{m_kn_k} $ then  
\begin{equation*}
\FF(t) = \begin{bmatrix}
\widetilde\FF(t)\\
0
\end{bmatrix},
\end{equation*}
where
\begin{equation*}
\widetilde\FF(t) = 
\widetilde\FF(s_{m_1n_1}) \widetilde\FF(s_{m_2n_2}) \cdots \widetilde\FF(s_{m_kn_k}) .
\end{equation*}

The formal Taylor  series for the error $\E(h):=\y_1-\y(x_0+h)$, in a Runge--Kutta step, is found from (\ref{eq:RK})  and (\ref{eq:DE}).  Rewrite this in the form
\begin{equation}
\E(h) = \sum_{p=1}^{\infty} h^p\sum_{u\in U_p} 
\sum_{t\in u} \frac1{\sigma(t)}\Big(\Phi(t) - \frac1{t!}\Big)\FF(t).
\end{equation}
The term corresponding to $u\in U_p$ takes a special form for the scalar case $d=1$, because $\FF(t)$ is identical for all $t\in u$ and the order conditions for the $h^p$ terms become
\begin{equation}\label{eq:err_s}
\sum_{t\in u} \frac1{\sigma(t)}\Big(\Phi(t) - \frac1{t!}\Big)=0, \qquad u\in U_p,
\end{equation}
compared with
\begin{equation}\label{eq:err_v}
\Phi(t) = \frac1{t!},\qquad t\in T_p,
\end{equation}
in the general vector case.

For each $p$, let $m_p$ the number of $t\in T_p$, and define $M_p=\sum_{i=1}^pm_p$.  Thus, $M_p$ is the number of order conditions for a vector problem.
Similarly, let $n_p$ be the number of members of $U_p$ and  $N_p=\sum_{i=1}^pn_p$
so that $N_p$ is the number of order conditions for a scalar problem.
Values of $m_p$, $M_p$, $n_p$ and $N_p$ are shown for $p\le 20$ in Table \ref{Table:OrderCond2}.
Note that neither  $(n_p)$ nor $(N_p)$ is listed in the On-Line Encyclopedia of Integer Sequences at \url{https://oeis.org/}.
For further details on how to compute those values (for larger values of $p$ without generating all trees) we refer to 
our Mathematica and Julia code available through \url{https://github.com/computational-b-series/isomers}.

\begin{table}[htp]
	\caption{\label{Table:OrderCond2}
		Number of scalar and vector order conditions for order $p$}
	\[
		\begin{array}{rrrrr}
			p  & m_p       & M_p        & n_p    & N_p    \\\hline
			1  & 1         & 1          & 1      & 1      \\
			2  & 1         & 2          & 1      & 2      \\
			3  & 2         & 4          & 2      & 4      \\
			4  & 4         & 8          & 4      & 8      \\
			5  & 9         & 17         & 8      & 16     \\
			6  & 20        & 37         & 15     & 31     \\
			7  & 48        & 85         & 28     & 59     \\
			8  & 115       & 200        & 51     & 110    \\
			9  & 286       & 486        & 91     & 201    \\
			10 & 719       & 1205       & 160    & 361    \\
			11 & 1842      & 3047       & 278    & 639    \\
			12 & 4766      & 7813       & 475    & 1114   \\
			13 & 12486     & 20299      & 803    & 1917   \\
			14 & 32973     & 53272      & 1342   & 3259   \\
			15 & 87811     & 141083     & 2218   & 5477   \\
			16 & 235381    & 376464     & 3629   & 9106   \\
			17 & 634847    & 1011311    & 5885   & 14991  \\
			18 & 1721159   & 2732470    & 9455   & 24446  \\
			19 & 4688676   & 7421146    & 15068  & 39514  \\
			20 & 12826228  & 20247374   & 23824  & 63338  
		\end{array}
	\]
\end{table}

\section{Construction of methods of order 5 and 6}\label{sect:o56}

Given that the number of order conditions is lower for scalar problems than for vector problem, one might think that restricting to scalar problems simplifies the construction of Runge--Kutta method of order $p$ considerably, as observed in the autonomous, scalar case, cf. \cite{MR1957156}.  It turns out, however, at least for $p=5$ und $p=6$, that this is not the case. The following theorem indicates one reason.

\begin{theorem}\label{Th:d1}
	A method of order $p$ for \eqref{eq:scalar} with $p\le 6$
	that satisfies the simplifying condition {\rm D(1)} is also of the same order for \eqref{eq:vector}.
\end{theorem}
\begin{proof}
	There is nothing to prove for $p\le 4$. 
	For $p\in\{5,6\}$,  it follows by inspecting the isomeric classes:
	It turns out, that in every class there is never more than one 
	tree for which D(1) cannot be applied.
\end{proof}

A method that has different orders of convergence for \eqref{eq:scalar} and \eqref{eq:vector} will be called \emph{ambiguous}.
Theorem~\ref{Th:d1} states that a method of ambiguous order with $p\le 6$ cannot satisfy D(1).

\def\squash{0.6}
\begin{remark}
	Note, that the proof of Theorem~\ref{Th:d1} is not applicable for $p\ge 7$, since, for example, 
	one has the isomeric class 
	$\{t_{62}, t_{65}, t_{76}, t_{81}\}$ containing in particular
	\[
	t_{62} = \T{62}\sim\T{65} = t_{65}
	\]
but {\rm D(1)} cannot be applied to either of these trees.
\end{remark}

\subsection{Deriving specific methods using systematic searches}
In this section, the degree of a tree $\delta(t)$, is defined recursively by
\begin{align*}
\delta(\tau) &= 0,\\
\delta(\tau_m t_1 t_2 \cdots t_m) &= 1 + \sum_{i=1}^m \delta(t_i).
\end{align*}
For ambiguous order five, the following conditions need to be satisfied:
\begin{alignat*}{7}
\delta(t) &=0, &b\tr1 &=1,\\
\delta(t) &=1, &b\tr c &=\tfrac12, &b\tr c^2 
&=\tfrac13,&b\tr c^3&=\tfrac14, \\
&&b\tr c^4 &=\tfrac15,\\
\delta(t) &=2, &b\tr Ac &=\tfrac16, &b\tr cAc &=\tfrac18,&b\tr Ac^2 &=\tfrac1{12},\\
&&b\tr c^2Ac &=\tfrac1{10},  
&b\tr cAc^2 &=\tfrac1{10},& b\tr Ac^3 &=\tfrac1{20}\\
\delta(t) &=3, &b\tr A^2c &=\tfrac1{24}, 
&b\tr cA^2c &=\tfrac1{30}+g, 
\qquad
&b\tr AcAc &=\tfrac1{40}-g, \\
&&\qquad b\tr A^2c^2 &=\tfrac1{60},\qquad&b\tr (Ac)^2 &=\tfrac1{20},\\
\delta(t) &=4, 
 &b\tr A^3c &=\tfrac1{120}, \\
\end{alignat*}
where $g\ne 0$.

A test program was constructed which, for given $p$ and $q$, evaluated the
discrepancy in all order conditions, such that $|t|\le p$ and $\delta(t)\le q$,
and carried out a solution for a selected subset of the equations  provided by
the non-zero discrepancies.  We will refer to an application of the test
procedure as $\text{\sf test}(p,q)$

\def\arraystretch{1.5}

\subsection{A method of ambiguous order 5}
The  method given by (\ref{eq:fake5}), was derived by carrying out a number of steps:
\begin{enumerate}
\item Choose $[c_2,c_3,c_4,c_6]=[\frac14,\frac12,\frac34,1]$, on the basis of simplicity.
\item Carry out $\text{\sf test}(4,2)$ and solve for $b\tr$ with $c_5$ as a free parameter.
\item Carry out $\text{\sf test}(5,2)$ and solve for $c_5$.
\item Carry out $\text{\sf test}(5,3)$ and solve the resulting linear equations in the free $a_{ij}$.
\item Carry out $\text{\sf test}(5,4)$ and select those equations which can be written linearly,
in terms of free $a_{ij}$ and $g$.
\item Repeat 5. until $\text{\sf test}(5,4)$ results in no further unsatisfied equations.
\end{enumerate}
\begin{equation}\label{eq:fake5}
	\begin{array}{c|cccccc}
		0          &                                                                                                         \\
		\frac14    & \frac{1}{4}                                                                                             \\
		\frac12    & -\frac{1}{2}\m     & 1                                                                                    \\
		\frac34    & \frac{3}{16}     & 0               & \frac{9}{16}                                                       \\
		\frac3{10} & \frac{291}{2500} & \frac{108}{625} & \frac{63}{2500} & -\frac{9}{625} \m                                  \\
		1          & -\frac{146}{135}\m & \frac{152}{15}  & -\frac{7}{15}\m  & \frac{428}{405} & -\frac{700}{81} \m               \\\cline{1-7}
		           & \frac{5}{54}     & 0               & 0               & \frac{32}{81}   & \frac{250}{567} & \frac{1}{14}
	\end{array}
\end{equation}
For the trees $t_{12}=\T{12}$ and $t_{15}=\T{15}$, method \eqref{eq:fake5} has 
\begin{equation}\label{eq:o5}
\left(\Phi(t_{12}) - \frac{1}{t_{12}!}  \right) + 
\left(\ \Phi(t_{15})- \frac{1}{t_{15}!}\right) = 0,
\end{equation}
but the terms do not vanish individually.
All other order conditions for the trees up to order $5$ are satisfied. 
Hence, method \eqref{eq:fake5} is of order 5 for \eqref{eq:scalar} but only of order 4 for \eqref{eq:vector}.  
Note that  factors $\sigma(t_{12})^{-1}$, $\sigma(t_{15})^{-1}$, are omitted from the terms in (\ref{eq:o5}), 
because these are identical.

\subsection{A method of ambiguous order 6}

The following method was constructed, by a similar process to (\ref{eq:fake5}).
\def\A#1/#2.{\frac{#1\sqrt{415}}{#2}}
\def\F#1/#2.{\frac{#1}{#2}}
\def\SP{\hspace{-1pt}}
\begin{equation}\label{eq:fake6}
{\scriptsize
\hspace{-10mm}\begin{array}{@{}c@{}|@{}c@{\SP}c@{\SP}
c@{\SP}c@{\SP}c@{\SP}c@{}c@{\SP}c@{\SP}}
0\\
\F1/2.&\F1/2.\\
1 & 0 & 1\\
\F1/5. &\F16/125.&\F13/125.&\F\!-\!4/125.\\
\F2/5.&\A\!-\!136\!-\!2/2875.&\A\!-\!638\!-\!6/2875.&\A821\!+\!7/11500.&\A275\!+\!/460.\\
\F3/5.&\A2469\!-\!31/40250.&\A\!-\!777\!+\!18/2875.&\A18979\!-\!326/241500.&
\A895\!+\!2/3220.&\A95\!-\!1/210.\\
1&\A91295\!-\!6701/45885.&\A\!-\!1415\!+\!12/2185.&\A\!-\!42285\!+\!3300/6118.&
\A\!-\!42285\!+\!3300/6118.&\A4115\!-\!285/399.&\A\!-\!260\!+\!20/57.\\
\F4/5.&\A\!-\!77534\!+\!6878/181125.&
\A1116\!-\!12/2875.&\A\!-\!3216783\!+\!103963/14852250.&\A3489\!-\!223/1610.&
\A\!-\!918\!+\!58/315.&\A82\!-\!4/45.&\A\!-\!285\!+\!38/15375.\\
\hline
&\frac{19}{288}&0&0&\frac{25}{96}&\frac{25}{144}&\frac{25}{144}&\frac{19}{288}&\frac{25}{96}
\end{array}}\hspace{-10mm}
\end{equation}

For $t_{25}=\T{25}$ and $t_{26}=\T{26}$, one finds
\begin{align*}
	\phi(t_{25}) - \frac{1}{t_{25}!} & = \frac{-20-3 \sqrt{415}}{82800}, \\
	\phi(t_{26}) - \frac{1}{t_{26}!} & = \frac{-20-3 \sqrt{415}}{41400},
\end{align*}
whereas the corresponding values for the isomeric forms $t_{31}=\T{31}$ and $t_{32}=\T{32}$ 
take the opposite signs. Since  also $\sigma(t_{25})=\sigma(t_{31})$ and $\sigma(t_{26})=\sigma(t_{32})$,
method \eqref{eq:fake6} is of order 6 for scalar problems but only order 5 for vector problems.

\section{Numerical experiments}\label{sect:exp}

Following~\cite{bseriesbook}, we consider the scalar initial value problem
\begin{equation}\label{eq:example-scalar}
	\frac{\dd y}{\dd x} = \frac{y-x}{y+x},\quad y(x_0) = y_0,\quad x\in[x_0, x_1],
\end{equation}
with $t_0 = \exp(\frac{\pi}{10})$, $t_1 = \exp(\frac{\pi}{2})$, $x_0 = t_0 \sin(\ln(t_0))$
and $x_1 = t_1 \sin(\ln(t_1))$
and the equivalent vector problem
\begin{equation}\label{eq:example-vector}
	\frac{\dd}{\dd t} \begin{bmatrix}z^1\\z^2\end{bmatrix} =
	\|z\|^{-1} \begin{bmatrix}z^2+z^1\\z^2-z^1 \end{bmatrix},\quad z(t_0) = [x_0, y_0]\tr,\quad t\in[t_0,t_1].
\end{equation}
Figure~\ref{Fig:errorplots} shows the norm of the error at the endpoint for the two methods \eqref{eq:fake5} and
\eqref{eq:fake6} applied to \eqref{eq:example-scalar} and \eqref{eq:example-vector}, respectively. 
It is seen that the methods behave as expected: the observed ``orders'' differ approximately by 
one.

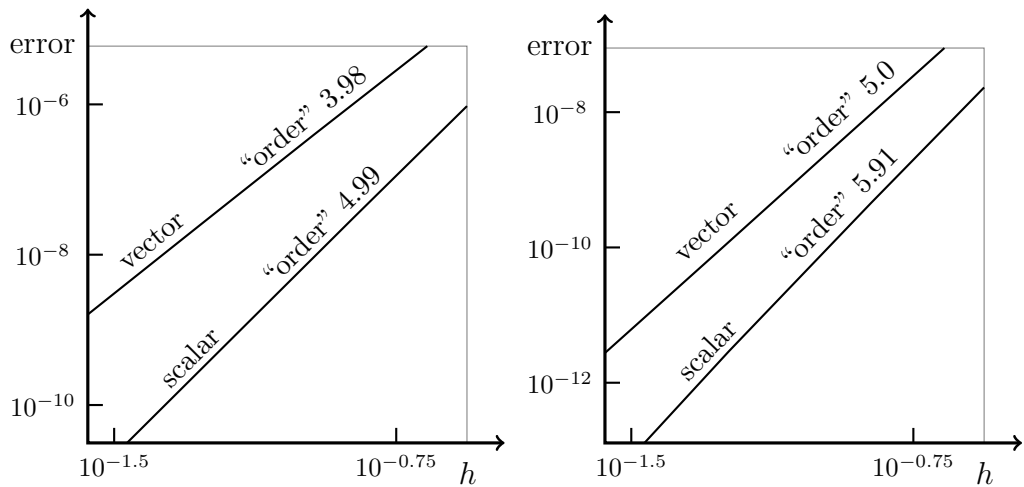
\begin{figure}[htp]
\centerline{
\begin{tikzpicture}[x=5cm,y=1cm]
\draw[gray](-1.570,-10.504)rectangle(-0.5619,-5.225);
\draw[<->,very thick]  (-1.570,-4.725) to(-1.570,-10.504)to (-0.4619,-10.504);
\draw[thick](-1.570,-8.795)to(-1.269,-7.595)to(-0.9684,-6.402)to(-0.6674,-5.225);
\draw[thick](-1.465,-10.50)to(-1.164,-9.002)to(-0.8629,-7.505)to(-0.5619,-6.023);
\draw[thick](-1.5,-10.504)to ($(-1.5,-10.504)+(0,2mm)$); 
\node () at ($(-1.5,-10.504)-(0,3mm)$)  {\footnotesize$10^{-1.5}$}; 
\draw[thick](-0.75,-10.504)to ($(-0.75,-10.504)+(0,2mm)$); 
\node () at ($(-0.75,-10.504)-(0,3mm)$)  {\footnotesize$10^{-0.75}$}; 
\node () at ($(-0.56,-10.504)-(0,4mm)$)  {$h$}; 
\draw[thick](-1.570,-6)to ($(-1.570,-6)+(2mm,0)$); 
\node () at ($(-1.570,-6)-(6mm,0)$)  {\footnotesize$10^{-6}$}; 
\draw[thick](-1.570,-8)to ($(-1.570,-8)+(2mm,0)$); 
\node () at ($(-1.570,-8)-(6mm,0)$)  {\footnotesize$10^{-8}$}; 
\draw[thick](-1.570,-10)to ($(-1.570,-10)+(2mm,0)$); 
\node () at ($(-1.570,-10)-(6mm,0)$)  {\footnotesize$10^{-10}$}; 
\node () at ($(-1.570,-5.225)-(6mm,0)$)  {error}; 
\node[draw=none,rotate=39.5]() at (-1.0,-6.2)  {\small``order" $3.98$};
\node[draw=none,rotate=39.5]() at (-1.4,-7.8)  {\small vector};
\node[draw=none,rotate=45]() at (-0.95,-7.6)  {\small``order" $4.99$};
\node[draw=none,rotate=45]() at (-1.3,-9.35) {\small scalar};
\end{tikzpicture}
\begin{tikzpicture}[x=5cm,y=0.9cm]
\draw[gray](-1.570,-12.89)rectangle(-0.5619,-7.052);
\draw[<->,very thick]  (-1.570,-6.552) to(-1.570,-12.89)to (-0.461,-12.89);
\draw[thick](-1.570,-11.56)to(-1.269,-10.06)to(-0.9684,-8.555)to(-0.6674,-7.052);
\draw[thick](-1.465,-12.89)to(-1.164,-11.11,)to(-0.8629,-9.345)to(-0.5619,-7.635);
\draw[thick](-1.5,-12.89)to ($(-1.5,-12.89)+(0,2mm)$); 
\node () at ($(-1.5,-12.89)-(0,3mm)$)  {\footnotesize$10^{-1.5}$}; 
\draw[thick](-0.75,-12.89)to ($(-0.75,-12.89)+(0,2mm)$); 
\node () at ($(-0.75,-12.89)-(0,3mm)$)  {\footnotesize$10^{-0.75}$}; 
\node () at ($(-0.56,-12.89)-(0,4mm)$)  {$h$}; 
\draw[thick](-1.570,-8)to ($(-1.570,-8)+(2mm,0)$); 
\node () at ($(-1.570,-8)-(6mm,0)$)  {\footnotesize$10^{-8}$}; 
\draw[thick](-1.570,-10)to ($(-1.570,-10)+(2mm,0)$); 
\node () at ($(-1.570,-10)-(6mm,0)$)  {\footnotesize$10^{-10}$}; 
\draw[thick](-1.570,-12)to ($(-1.570,-12)+(2mm,0)$); 
\node () at ($(-1.570,-12)-(6mm,0)$)  {\footnotesize$10^{-12}$}; 
\node () at ($(-1.570,-7.052)-(6mm,0)$)  {error}; 
\node[draw=none,rotate=42]() at (-0.95,-8.0)  {\small``order" $5.0$};
\node[draw=none,rotate=42]() at (-1.3,-9.75)  {\small vector};
\node[draw=none,rotate=46]() at (-0.95,-9.4)  {\small``order" $5.91$};
\node[draw=none,rotate=46]() at (-1.3,-11.5) {\small scalar};
\end{tikzpicture}
}
\caption{\label{Fig:errorplots}Error behaviour of methods of ambiguous order applied to  the scalar 
problem (\ref{eq:example-scalar}) and the vector problem (\ref{eq:example-vector}).  
The methods are  the ambiguous fifth order method (\ref{eq:fake5}) (left-hand figure) and 
the ambiguous sixth order method (\ref{eq:fake6}) (right-hand figure).
The numerically observed ``order" is attached to each result.}
\end{figure}

We do not claim that the new methods~\eqref{eq:fake5} and \eqref{eq:fake6} are superior for any practical calculation.
The sole purpose of constructing them was to show that methods with ambiguous order do exist.

\section{Conclusion}\label{sect:conc}
By using factorisation of trees into atomic stumps, and allowing for possible permutations of factors, the concept of isomeric classes of trees is introduced.
It is shown that, for each class, only a single order condition needs to hold, for scalar problems, in place of 
a separate condition for each tree in the class, as is required for high dimensional problems.  

Using this analysis, special methods with \emph{ambiguous} order, 5 and 6 respectively,
have been derived.  It was expected that the asymptotic error behaviour would drop by one order when the methods were applied to problems of dimension greater than one.  This was observed to be as predicted in numerical experiments.


\bibliographystyle{elsarticle-num} 

\end{document}